\documentclass[11pt,a4paper]{article}

\usepackage{palatino}
\usepackage{graphicx}
\usepackage{faktor}
\usepackage{amsmath}
\usepackage{amsfonts}
\usepackage{amsthm}
\usepackage{mathrsfs}
\usepackage{latexsym}
\usepackage{amssymb}
\usepackage{amscd}
\usepackage{tikz-cd}
\usepackage[all]{xy}
\usepackage{authblk}
\usepackage{hyperref}
\usepackage{enumerate}
\usepackage{multicol}

\usepackage[scale=0.75]{geometry}

\usepackage{stackrel}
\usepackage{tensor}  
\usepackage[fleqn,tbtags]{mathtools,nccmath} 

\numberwithin{equation}{section}

\newtheorem{defn}[equation]{Def{i}nition}
\newtheorem{prop}[equation]{Proposition}
\newtheorem{lemma}[equation]{Lemma}
\newtheorem{theorem}[equation]{Theorem}

\theoremstyle{remark}
\newtheorem{remark}[equation]{Remark}

\theoremstyle{definition}
\newtheorem{example}[equation]{Example}

\DeclareSymbolFont{AMSb}{U}{msb}{m}{n}
\DeclareMathSymbol{\N}{\mathbin}{AMSb}{"4E}
\DeclareMathSymbol{\Z}{\mathbin}{AMSb}{"5A}
\DeclareMathSymbol{\R}{\mathbin}{AMSb}{"52}
\DeclareMathSymbol{\Q}{\mathbin}{AMSb}{"51}
\DeclareMathSymbol{\I}{\mathbin}{AMSb}{"49}
\DeclareMathSymbol{\C}{\mathbin}{AMSb}{"43}

\newcommand{\tn}[1]{\textnormal{#1}}

\providecommand{\keywords}[1]
{
  \small	
  \textbf{\textit{Keywords:}} #1
}

\title{Quotient bifinite extensions and the finitistic dimension conjecture}

\author[a]{John William MacQuarrie}
\author[b]{Fernando dos Reis Naves}
\affil[a]{Universidade Federal de Minas Gerais, Belo Horizonte, MG, Brazil}

\begin{document}

\footnotetext{\textit{Email addresses:} john@mat.ufmg.br (John MacQuarrie), 
fernando.r.naves@gmail.com (Fernando dos Reis Naves)}


\maketitle

\begin{abstract}
We prove that if $B\subseteq A$ is an extension of finite dimensional algebras such that the projective dimension of $A/B$ as a $B$-bimodule is finite, if $A$ has finite finitistic dimension, then so does $B$.  We exhibit examples demonstrating that the algebra $B$ appearing in such an extension can be more complicated than $A$.
\end{abstract}


\section{Introduction}

Let $A$ be a finite dimensional associative algebra over a field. The (small) finitistic dimension $\mathrm{fin.dim}A$ of $A$ is defined to be the supremum of the projective dimensions of the finitely generated left $A$-modules having finite projective dimension. The famous finitistic dimension conjecture (for finite dimensional algebras over a field) asserts that the finitistic dimension of an arbitrary finite dimensional algebra is finite. 
The conjecture is over 62 years old and remains wide open, being proved only for special classes of algebra -- see for instance \cite{green91,@igusa2015,Wang1994} for some cases where the conjecture is known to hold (there are many others). 

An \emph{extension} of finite dimensional algebras is simply a finite dimensional algebra $A$ with unital subalgebra $B$, denoted $B\subseteq A$.  A fruitful approach to the finitistic dimension conjecture, pioneered by Xi and coauthors, 
is to compare the finitude of the finitistic dimensions of $A$ and $B$ when the extension is assumed to have certain properties.  For instance, suppose that the radical of $B$ is a left ideal of $A$.  If either the representation dimension of $A$ is at most three \cite[Theorem 4.2]{xi2006}, or $A$ has finite projective dimension as a right $B$-module and $\mathrm{fin.dim}A$ is finite \cite[Corollary 1.4]{xi2013}, then $\mathrm{fin.dim}B$ is finite.



Of particular relevance to the current discussion is \cite[Theorem 6.14]{iusenko2021}: if $B\subseteq A$ is an extension of algebras such that $i)$ $A/B$ has finite projective dimension as a $B$-bimodule; $ii)$ $A/B$ is projective as either a left or a right module; $iii)$ some tensor power (over $B$) of $A/B$ is projective as a $B$-bimodule and $iv)$ the $B$-relative projective dimension of the bimodule $A$ is finite, then $A$ has finite finitistic dimension if, and only if, $B$ does.  We show here that in order to prove the ``downwards direction'', these hypotheses can be weakened considerably.  Say that the extension $B\subseteq A$ is \emph{quotient bifinite} if $A/B$ has finite projective dimension as a $B$-bimodule.  We prove 

\begin{theorem} \label{theorem main}
Let $B \subseteq A$ be a quotient bifinite extension. If $\mathrm{fin.dim}A < \infty$, then $\mathrm{fin.dim}B< \infty$.
\end{theorem}

We then give some simple examples which show that there are quotient bifinite extensions for which $A$ clearly has finite finitistic dimension, while the subalgebra $B$ is more complex.

\medskip

\textbf{Acknowledgements.} We thank the referee for their careful reading of the text and helpful comments, which have improved the exposition. The first author was partially supported by CNPq Universal Grant 402934/2021-0, CNPq Produtividade 1D grant 303667/2022-2, and FAPEMIG Universal Grant APQ-00971-22.  The second author was supported by a CAPES doctoral grant -- Finance Code 001.

\section{Preliminaries}

We fix some notation.  Throughout our discussion, $A$ is a finite dimensional $k$-algebra for $k$ a field, and $B$ is a unital subalgebra of $A$. 
Given a left $A$-module $M$, we denote by $\mathrm{pd}(_AM)$ the projective dimension of $M$, and by $\Omega_A^n(M)$ the $n$-th syzygy of $M$ $($setting $\Omega_A^0(M) = M )$.

The \textbf{finitistic dimension} of the algebra $A$ is
$$
\mathrm{fin.dim}A := \mathrm{sup}\{ \mathrm{pd}(_AM): M \hbox{ a finitely generated left $A$-module with }  \mathrm{pd}(_AM) < \infty \}
$$
and the \textbf{global dimension} of $A$ is
$$
\mathrm{gl.dim}A := \mathrm{sup}\{ \mathrm{pd}(_AM): M \hbox{ a finitely generated left $A$-module}\}.
$$

\begin{lemma}\label{lemma tensor sequence is exact}
Let
$$
0 \rightarrow P_n \xrightarrow{\varphi_n} P_{n-1} \xrightarrow{\varphi_{n-1}} \cdots \rightarrow P_0 \xrightarrow{\varphi_0} N \rightarrow 0
$$
be a projective resolution of the left $A$-module $N$. If $M$ is a right $A$-module such that $\mathrm{Tor}^A_j(M,N)=0$ for all $j \geqslant 1$, then the sequence
$$
 0 \rightarrow  M \otimes P_n \rightarrow M \otimes P_{n-1}  \rightarrow \cdots \rightarrow M \otimes  P_0  \rightarrow M \otimes N \rightarrow 0
$$
is exact. 
\end{lemma}

\begin{proof}
This follows easily from the definition of $\tn{Tor}_j^A(M,N)$.
\end{proof}

The following is well-known (cf.\ \cite[Proposition 4.7]{assem2006}):

\begin{lemma}\label{lemma pd bounds}
Let $0 \rightarrow L \rightarrow M \rightarrow N \rightarrow 0$ be a short exact sequence of $A$-modules. Then
\begin{itemize}
    \item [$(i)$] $\mathrm{pd}(_AN) \leqslant \mathrm{sup} \bigl\{ \mathrm{pd}(_AL) +1, \mathrm{pd}(_AM)  \bigr\}$;
    \item [$(ii)$]$\mathrm{pd}(_AM) \leqslant \mathrm{sup} \bigl\{ \mathrm{pd}(_AL), \mathrm{pd}(_AN)  \bigr\}$;
    \item [$(iii)$] $\mathrm{pd}(_AL) \leqslant \mathrm{sup} \bigl\{ \mathrm{pd}(_AM), \mathrm{pd}(_AN) -1  \bigr\}$.
\end{itemize}
\end{lemma}

\begin{lemma}\label{lemma pd of A tensor Omega m}
If $\mathrm{pd}(A_B) = n < \infty$ then, for any left $B$-module $M$ and  $m \geqslant n, $ 
$$\mathrm{pd}\bigl(_A A \otimes_B \Omega_B^m(M)\bigr)\leqslant \mathrm{pd}(_BM).$$ 
\end{lemma}

\begin{proof}
The result is trivial if $\mathrm{pd}(_BM)$ is infinite, so assume it is finite.  Let 
$$
0 \rightarrow P_s \xrightarrow{\varphi_s} P_{s-1} \xrightarrow{\varphi_{s-1}} \cdots \rightarrow P_0 \xrightarrow{\varphi_0} \Omega_B^m(M) \rightarrow 0
$$
be a projective resolution of the $B$-module $\Omega_B^m(M)$. Since 
$m \geqslant n$, we have 
$$\mathrm{Tor}^B_j(A,\Omega^m_B(M)) = \mathrm{Tor}^B_{j+m}(A,M) = 0$$
for all $j \geqslant 1$. Therefore, by Lemma \ref{lemma tensor sequence is exact}, we have the following exact sequence
$$
0 \rightarrow A \otimes_B P_s \rightarrow A \otimes_B P_{s-1} \rightarrow \cdots \rightarrow A \otimes_B P_{0} \rightarrow A \otimes_B \Omega^m_B(M) \rightarrow 0.
$$
The modules $A\otimes_B P_i$ are projective $A$-modules, because the $P_i$ are projective $B$-modules. Hence
the sequence obtained is a projective resolution of $A \otimes_B \Omega^m_B(M)$.  Thus 
$$\mathrm{pd}\bigl(_A A \otimes_B \Omega^m_B(M) \bigl) \leqslant \mathrm{pd}\bigl(_B \Omega^m_B(M) \bigr) \leqslant \mathrm{pd}(_B M).$$
\end{proof}

\begin{lemma}\label{lemma bound pd of restriction}
If $\mathrm{pd}(_BA) < \infty$, then for any $A$-module $X$ with $ \mathrm{pd}(_A X)< \infty$ we have
$$
\mathrm{pd}(_B X) \leqslant \mathrm{pd}(_A X) + \mathrm{pd}(_B A).
$$
\end{lemma}
\begin{proof}
The projective dimension of a projective $A$-module, viewed as a $B$-module, is finite, because $\mathrm{pd}(_BA) < \infty$.
Let 
$$
0 \rightarrow P_s \rightarrow P_{s-1} \rightarrow \cdots \rightarrow P_0 \rightarrow X \rightarrow 0
$$
be a minimal projective resolution of ${}_AX$. We have by \cite[Lemma 2.3]{xi2006} that
$$
\mathrm{pd}(_BX) \leqslant s + \mathrm{sup}\{ \mathrm{pd}(_BP_i): i \in \{0,\ldots,s\} \} \leqslant \mathrm{pd}(_AX) + \mathrm{pd}(_BA).
$$
\end{proof}

\begin{lemma}[{\cite[Lemma 6.3]{iusenko2021}}]\label{lemma induced module projective}
If $P$ is a projective $B$-bimodule and $X$ is any left $B$-module, then $P\otimes_BX$ is projective as a left $B$-module.
\end{lemma}

\section{Quotient bifinite extensions}


\begin{defn}
The extension of algebras $B \subseteq A$ is \textbf{quotient bifinite} if the projective dimension of $A/B$ as a $B$-bimodule is finite.
\end{defn}

\begin{remark}\label{remark pd bounds} 
If $B \subseteq A$ is quotient bifinite, then the projective dimension of $A/B$ as a right and as a left $B$-module is bounded above by the projective dimension of $A/B$ as a $B$-bimodule. The short exact sequence
$$
0 \rightarrow B \hookrightarrow A \rightarrow A/B \rightarrow 0
$$
and Lemma \ref{lemma pd bounds} show that  $\textrm{pd}(_BA)\leqslant \textrm{pd}(_BA/B)$ and $\textrm{pd}(A_B)\leqslant \textrm{pd}(A/B_B)$.
\end{remark}

\begin{proof}(of Theorem \ref{theorem main})
Denote by $n_r$ the projective dimension of $A/B$ as a right $B$-module and by $n_b$ the projective dimension of $A/B$ as a $B$-bimodule. Let $M$ be a left $B$-module with finite projective dimension. The sequence of right $B$-modules 
$$
0 \rightarrow B \hookrightarrow A \rightarrow A/B \rightarrow 0
$$ 
induces the exact sequence
$$
\mathrm{Tor}^B_1(A/B,\Omega^{n_r}_B(M))
\rightarrow 
\Omega^{n_r}_B(M) \rightarrow A \otimes_B \Omega^{n_r}_B(M) \rightarrow (A/B)\otimes_B \Omega^{n_r}_B(M) \rightarrow 0.
$$
But the projective dimension of $A/B$ as a right $B$-module is $n_r$, so
$$\mathrm{Tor}^B_1(A/B,\Omega^{n_r}_B(M)) = \mathrm{Tor}^B_{n_r+1}(A/B,M) = 0$$ 
and hence we have an exact sequence
$$
0 \rightarrow  \Omega^{n_r}_B(M) \rightarrow A \otimes_B \Omega^{n_r}_B(M) \rightarrow (A/B)\otimes_B \Omega^{n_r}_B(M) \rightarrow 0.
$$

Since $M$ has finite projective dimension, so does $A \otimes_B \Omega^{n_r}_B(M)$ as a left $A$-module by Lemma \ref{lemma pd of A tensor Omega m}. Hence $\mathrm{pd}\bigl(_A A \otimes_B \Omega^{n_r}_B(M) \bigr) \leqslant \mathrm{fin.dim}(A)$. It follows from Lemma \ref{lemma bound pd of restriction} that 
\begin{ceqn}
\begin{align}\label{5.1}
    \mathrm{pd}(_B A\otimes_B \Omega^{n_r}_B(M)) \leqslant  \mathrm{fin.dim}(A) + \mathrm{pd}(_BA).
\end{align}
\end{ceqn}

Let 
$$
0 \rightarrow Q_m \rightarrow \cdots \rightarrow Q_1 \rightarrow Q_0 \rightarrow A/B \rightarrow 0
$$
be a finite projective resolution of $A/B$ as a $B$-bimodule. Since the modules $Q_i$ are projective as right $B$-modules, it follows from (the left-right dual version of) Lemma \ref{lemma tensor sequence is exact} that the sequence
$$
0 \rightarrow Q_m \otimes_B \Omega^{n_r}_B(M)  \rightarrow \cdots \rightarrow Q_1 \otimes_B  \Omega^{n_r}_B(M)  \rightarrow Q_0 \otimes_B   \Omega^{n_r}_B(M)  \rightarrow (A/B)\otimes_B \Omega^{n_r}_B(M)  \rightarrow 0
$$
is exact. But the modules $Q_i \otimes_B  \Omega^{n_r}_B(M) $ are projective as left $B$-modules by Lemma \ref{lemma induced module projective}, and hence $(A/B)\otimes_B \Omega^{n_r}_B(M)$ has projective dimension as a left $B$-module not more than the projective dimension of $A/B$ as a $B$-bimodule, i.e. 
\begin{ceqn}
\begin{align}\label{5.2}
\mathrm{pd}(_B (A/B) \otimes_B \Omega^{n_r}_B(M)) \leqslant  n_b.
\end{align}
\end{ceqn}
As a result, we have 
$$
\begin{array}{cccc}
 \mathrm{pd}(_B M) & \leqslant & n_r + \mathrm{pd}(_B \Omega^{n_r}_B(M)) &  \\
 & \leqslant &  n_r + \mathrm{sup} \big\{ \mathrm{pd}(_B A \otimes_B \Omega^{n_r}_B(M)), \mathrm{pd}(_B (A/B)\otimes_B \Omega^{n_r}_B(M)) -1 \big\} & \hbox{Lemma }\ref{lemma pd bounds}  \\
 & \leqslant &  n_r +  \mathrm{sup}\big\{\mathrm{fin.dim}(A) + \mathrm{pd}(_BA), \mathrm{pd}(_B (A/B)\otimes_B \Omega^{n_r}_B(M)) -1 \big\} &  (\ref{5.1})  \\
 & \leqslant & n_r  + \mathrm{sup}\big\{\mathrm{fin.dim}(A) + \mathrm{pd}(_BA), n_b -1 \big\} &  \ (\ref{5.2}) \\
 & \leqslant & n_b  + \mathrm{sup}\big\{\mathrm{fin.dim}(A) + n_b, n_b -1 \big\} &   \hbox{Remark }\ref{remark pd bounds}   \\
 & = & 2n_b + \mathrm{fin.dim}(A). &  
\end{array}
$$
The number $2n_b + \mathrm{fin.dim}(A)$ is independent of $M$, so we are done.
\end{proof}






We present some examples.  The finitistic dimension of a finite dimensional monomial algebra is always finite by \cite{green91}, so we will give quotient bifinite extensions $B\subseteq A$ with $A$ monomial.  A result of Green and Marcos \cite[Theorem 6.1]{marco2017} says that to prove the finitistic dimension conjecture for finite dimensional algebras, it suffices to prove it for algebras whose Gabriel quiver has a directed path from any given vertex to any other, and so our examples will have this property.

\begin{example}
Let $A$ be the monomial algebra given by the following quiver with relations:
$$\begin{tikzcd}
	 & 2 \dar[dl, swap]{\alpha} &\\
	1\dar[dr, swap]{\delta}   & &  \dar[ul, swap]{\beta}  3  \\
	& 4 \dar[ur, swap]{\gamma} & 
\end{tikzcd}, \quad \delta \alpha \beta \gamma \delta  = 0.$$


The subalgebra $B$ of $A$ generated by the set $\{e_1,e_2,e_3,e_4, \alpha, \gamma, \delta, \varepsilon_1 = \beta \gamma, \varepsilon_2 = \alpha \beta \}$ has quiver with relations:
$$ \begin{tikzcd}
	 & 2 \dar[dl, swap]{\alpha} &\\
	1  \rar[rr]{\delta} & & \dar[dl]{\gamma} \dar[ul, swap]{\varepsilon_1 = \beta\gamma}  4  \\
	& 3 \dar[ul]{\varepsilon_2 = \alpha\beta} & 
\end{tikzcd} , \quad \alpha \varepsilon_1 - \varepsilon_2 \gamma = \delta \alpha \varepsilon_1 \delta =0.$$
Since $B$-bimodules of the form $X_{ij} = Be_i \otimes_k e_jB$ are projective, the sequence
$$
\begin{tikzcd}
0 \rar &  X_{14} \ar[dr] \rar[r]{\varphi_2}   & X_{13} \oplus X_{24}  \rar[r]{\varphi_1} \ar[dr]  &  X_{23}  \rar[r]{\varphi_0}  &  A/B\ar[r] & 0 \\
  &  &   \langle (e_1 \otimes \gamma, -\alpha \otimes e_4)  \rangle\ar[u, hook]     &   \langle \alpha \otimes e_3, e_2 \otimes \gamma  \rangle \ar[u, hook]  &         
\end{tikzcd}
$$
(where $\langle x\rangle$ indicates ``generated by $x$ as a $B$-bimodule'') is a projective resolution of the $B$-bimodule  $A/B$ (which has basis $\{\beta + B\}$), so the extension $B \subseteq A$ is quotient bifinite. Now since $A$ has finite finitistic dimension (being monomial) it follows from Theorem \ref{theorem main} that $\mathrm{fin.dim}B < \infty$.

\end{example}

\begin{example}
Let $A$ be the algebra given by the quiver with relations: 
$$
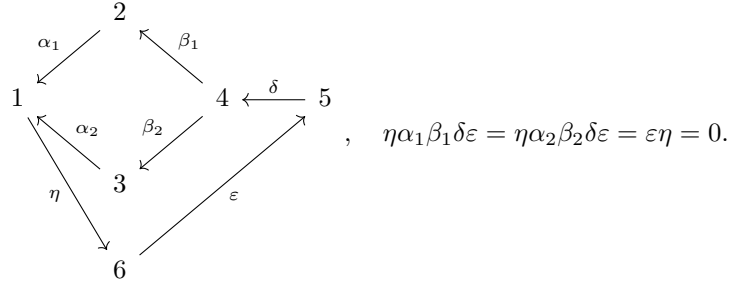
\begin{tikzcd}
   & 2 \dlar[swap]{\alpha_1} &  &    \\
  1 \rar[rdd,swap]{\eta}  &  & 4  \ular[swap]{\beta_1} \dlar[swap]{\beta_2} & 5 \lar[swap]{\delta}   \\
  & 3 \ular[swap]{\alpha_2} &  &  \\
   & 6 \dar[rruu,swap]{\varepsilon}  &  & 
\end{tikzcd}, \quad  \eta \alpha_1 \beta_1 \delta \varepsilon = \eta \alpha_2 \beta_2 \delta \varepsilon = \varepsilon \eta = 0.
$$
The subalgebra $B$ of $A$ generated by the set $$\bigl\{ e_1,e_2,e_3,e_4,e_5,e_6, \alpha_1,\alpha_2, \delta, \varepsilon, \eta, \theta_1 = \beta_1 \delta, \theta_2=\beta_2\delta, \omega=\alpha_1\beta_1 + \alpha_2 \beta_2 \bigr\},$$
has quiver with relations: 
$$
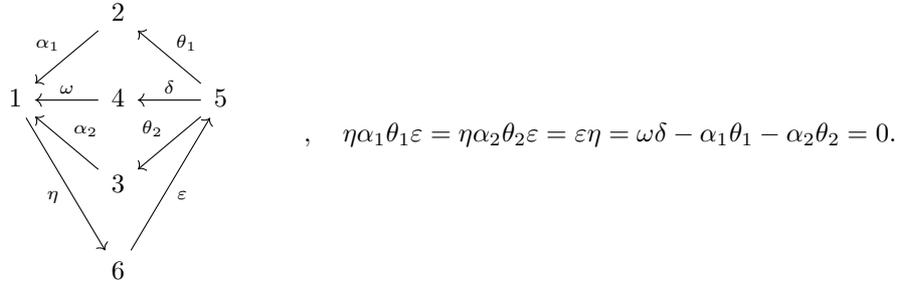
\begin{tikzcd}
   & 2 \dlar[swap]{\alpha_1} &  &    \\
  1 \rar[rdd,swap]{\eta}  &   4  \lar[swap]{\omega}  & 5 \lar[swap]{\delta} \dlar[swap]{\theta_2} \ular[swap]{\theta_1}   \\
  & 3 \ular[swap]{\alpha_2}  &  \\
   & 6 \dar[ruu,swap]{\varepsilon}   & 
\end{tikzcd}, \quad  \eta \alpha_1 \theta_1  \varepsilon = \eta \alpha_2 \theta_2  \varepsilon = \varepsilon \eta = \omega \delta - \alpha_1 \theta_1 - \alpha_2 \theta_2 = 0.
$$
The $B$-bimodule $A/B$ has basis
\[
\displaystyle  \bigl\{ \beta_1 + B, \beta_2 + B, \alpha_1 \beta_1 + B, \eta \alpha_1 \beta_1 + B \bigr\}. 
\]
The exact sequence
 %
$$
\begin{tikzcd}
 & 0 \rar & X_{15}  \rar & X_{14} \oplus X_{25} \oplus X_{35}  \rar  &  X_{34} \oplus X_{24}  \rar &  A/B \rar & 0  \\
\end{tikzcd}
$$
is a projective resolution of $A/B$ as a $B$-bimodule and hence $B \subseteq A$ is a quotient bifinite extension. So $\mathrm{fin.dim} B < \infty$ by Theorem \ref{theorem main}.
\end{example}

Similar constructions yield quotient bifinite extensions wherein $A$ is monomial and $B$ has relations that are linear combinations of arbitrarily many monomials.  The following generalizes \cite[Theorem 4.2]{Cibils_2022}.


\begin{prop}\label{prop gd preserved}
Let $B \subseteq A$ be a quotient bifinite extension. If $\mathrm{gl.dim}A < \infty$, then $\mathrm{gl.dim}B< \infty$.
\end{prop}
\begin{proof}
The proof is essentially identical to that of Theorem \ref{theorem main}: denote by  $g_A$ the global dimension of $A$ and by $n_b$ the projective dimension of $A/B$ as a $B$-bimodule. Let $M$ be a left $B$-module. As in the proof of Theorem \ref{theorem main}, $A \otimes_B \Omega_B^{n_b}(M)$  has projective dimension not more than $\mathrm{pd}(_B A) + g_A$ by Lemma \ref{lemma bound pd of restriction}. As in the proof of Theorem \ref{theorem main}, the left $B$-module $(A/B)\otimes_B \Omega_B^{n_b}(M)$ has projective dimension not more than the projective dimension of $A/B$ as a $B$-bimodule.  Hence 
$$
\mathrm{pd}(_B M) \leqslant {n_b} + \mathrm{sup} \{ {n_b}, \mathrm{pd}(_B A) + g_A \},
$$
so the global dimension of $B$ is bounded above by $2{n_b} + g_A$.
\end{proof}

The extension $k\subseteq k[x]/\langle x^2\rangle$ shows that the converse of Proposition \ref{prop gd preserved} is false.


\bibliography{mref}
\bibliographystyle{alpha}

\end{document}